\documentclass{amsart}
\usepackage{graphicx}
\usepackage{amssymb}

\usepackage{psfrag}

\usepackage[margin=15pt,font=small,labelfont=bf]{caption}

\usepackage{amsthm}

\newtheorem{thm}{Theorem}

\theoremstyle{definition}
\newtheorem*{defn}{Definition}

\theoremstyle{remark}
\newtheorem*{qn}{Question}

\newtheorem*{rem}{Remark}

\title{On the  genus of infinite groups}
\author{Iain Aitchison}
\address{Department of Mathematics and Statistics \\ University of Melbourne\\
Parkville 3010,
Australia}
\email{I.Aitchison@ms.unimelb.edu.au}
\curraddr{School of Mathematics, Statistics and Operations Research\\
Victoria University of Wellington\\
Wellington 6140\\
New Zealand}

\author{Lawrence Reeves}
\address{Department of Mathematics and Statistics \\ University of Melbourne\\
Parkville, 3010
Australia}
\email{lreeves@unimelb.edu.au}

\newcommand{\emp}[1]{\textit{\textbf{#1}}}

\newcommand{\<}{\langle}
\renewcommand{\>}{\rangle}
\renewcommand{\|}{\mid}
\newcommand{\Z}{\mathbb{Z}}

\begin{document}
\maketitle

\begin{abstract}
We associate to each finite presentation of a group $G$ a compact CW-complex that is a 3-manifold in the complement of a point, and whose fundamental group is isomorphic to $G$. We use this complex to define a notion of genus for $G$ and give examples, and also define a notion of `closed group'. A group has genus 0 if and only if it is the fundamental group of a compact orientable 3-manifold.  
\end{abstract}

\section{Introduction}

In \cite{B1} Bridson remarks that
`if all groups are fundamental
groups of manifold with additional structure, then one might hope to use that structure
to prove interesting facts about arbitrary finitely presented groups.'
In this spirit, we show that every finitely presented group is the fundamental group of a `singular 3-manifold'-- a complex obtained from a three manifold by coning one boundary component.
Given a finite presentation $P$ for $G$, there is a standard presentation
2-complex $K_P$ with $\pi_1(K_P) \cong G$. The complex $X_P$ we construct
has spine $K_P$, and the construction generalizes for $n\geqslant 3$ to produce $X^n_P$,
 a compact $n$-dimensional manifold homotopy equivalent to
$K_P$. For $n\geqslant 5$, taking the boundary gives the standard construction
of a closed $(n-1)$-manifold with fundamental group $G$.

In the following,  `compact manifold'   allows non-empty boundary;  `closed manifold' means compact with empty boundary. Generally,  manifolds will be orientable and compact, unless stated otherwise. 
By a \emp{singular 3-manifold}, we mean a 
compact complex that is a 3-manifold away from finitely many points, the links of which being connected
compact orientable surfaces. The meanings of `closed', `orientable' and  `compact' have the obvious extension to this context; in particular, the boundary of a compact singular 3-manifold is a closed surface with a collar neighbourhood.

This paper is arranged as follows. In the second section we describe a construction that takes a finite presentation and produces a singular 3-manifold with fundamental group given by the presentation and give a definition of genus. In section 3 we gives some examples of calculating the genus of a presentation. Another observation based on the constructed singular 3-manifold is given in section 4, and some remarks and questions are discussed in the final section.
 

We note that a different notion of genus for groups was introduced in \cite{W} and has been studied by several authors. 
With this notion all infinite groups are either of genus zero or of infinite genus. This is not the case with the notion we define.  We give examples of infinite groups that have genus one.

\section{The construction}

\subsection{Context}

We first recall some standard constructions of complexes and manifolds with specified fundamental group: 
Given a finite presentation $P$ having $n$ generators and $k$ relators, one can
construct the standard 2-complex $K_P$, with one vertex, $n$ (labelled) 1-cells, and $k$ 2-cells attached  according to the given relations. 
An account of this well-known construction can be found in, for example,
\cite[III.2]{LS}.
This suggests a notion of genus for a presentation defined by taking the genus of the graph given by the link of the vertex in the presentation 2-complex.

\begin{defn}
The \emp{link genus} of a finite presentation $P$ is the 
genus of the graph given by the link of the vertex in $K_P$. 
The \emp{link genus of a group $G$} is then given by minimising over all finite presentations of $G$.
\end{defn}

If $G$ is 3-manifold group, the link genus will be zero, but 
the converse is false.
We refine this notion of genus in order to obtain an if and only if statement. 

By general position arguments, $K_P$
can be embedded in 5-dimensional Euclidean space.
A regular neighbourhood of the embedded 2-complex is a compact 5-manifold with spine the 2-complex, and  with boundary a closed 4-manifold whose fundamental group is that of the 2-complex. To see this, note that the any loop in the 2-complex can be pushed off in 5-space, and hence projected to the 4-manifold; conversely, suppose a loop in the 4-manifold bounds a disc in the 5-manifold. This disc can be perturbed to be disjoint from the spine, and hence projects to the boundary, whose fundamental group thus projects surjectively and injectively to that of the 5-manifold, which is homotopy equivalent to the 2-complex. The 5-manifold has a handlebody decomposition with one 0-handles, $n$ 1-handles and $k$ 2-handles. 

Not all 2-complexes can be embedded in standard 4-space: $2+2<4$ fails with catastrophic consequences. Nonetheless we can emulate the handlebody construction, and abstractly build a 4-manifold-with-boundary. In this case, the 4-manifold and boundary 3-manifold are determined by a link (disjointly embedded circles), defined up to homotopy, and an additional choice of an integer for each 2-handle. A given presentation thus gives rise to infinitely many possibilities:
Standard handlebody theory in dimensions 3 and 4 enables us to represent the $n$ 1-handles by a collection of unknotted, unlinked oriented circles in the 3-sphere, representing free-group generators. 
The relations are determined by attached $k$ discs to $k$ circles,
chosen to give a link of $n+k$ total components in $S^3$.

The construction of bounded $m$ manifolds,   $m\geqslant 4$, with spine the 2-complex, breaks down for $m=3$: this is because in general we cannot find a disjoint collection of embedded circles on the boundary of a 3-dimensional handelbody along which to attach the 2-handles. 
However,   we can always find immersed circles representing 
the relations in a groups presentation: this motivates the following construction.

\subsection{Construction}

When referring to a finite presentation we will 
consider both the generating set and the set of defining relations as ordered sets.
Let $P_G= \langle x_1,x_2,\dots , x_n \ | \ r_1, r_2,\dots, \ r_k \rangle$ be a finite (ordered) presentation for a group $G$. 

One can find a collection of $k$ immersed circles on a
3-dimensional handlebody of genus $n$
representing the relations. The handlebody can be constructed by adding $n$ 1-handles to $2n$ disjointly embedded discs on the 2-sphere. The circles representing 
relations can be assumed to meet the boundaries of all 1-handles in a collection of disjointly embedded arcs. These and all remaining arcs 
can be thickened on the surface to give a collection of ribbons. Each of these remaining arcs can be considered a (possibly) immersed 2-dimensional 1-handle attached to the collection of discs. Now forget the 2-sphere, but keep the $2n$ discs, with all remaining ribbons made disjoint and attached to the discs, creating an orientable surface with boundary. Fill in all boundary components by attaching discs to obtain a closed orientable surface (possibly disconnected) which we shall denote $\Omega$. Cone this surface to a point $\omega$, attach 3-dimensional 1-handles to pairs of discs, and attach 3-dimensional 2-handles along the embedded ribbon neighbourhoods of circles representing relations. The resulting complex  deformation retracts to the 2-complex spine $K_P$.
This is illustrated in Figure \ref{fig:2}

 \begin{figure}[htb] 
   \centering
{\small
\psfrag{w}{$\omega$}
\psfrag{a}{Add 2-handles along closed ribbons}
\psfrag{b}{Add $n$ 1-handles}
\psfrag{c}{\parbox[t]{4cm}{Add discs to obtain \\ closed ribbon surface  $\Omega$}}
\includegraphics[width=0.5\textwidth]{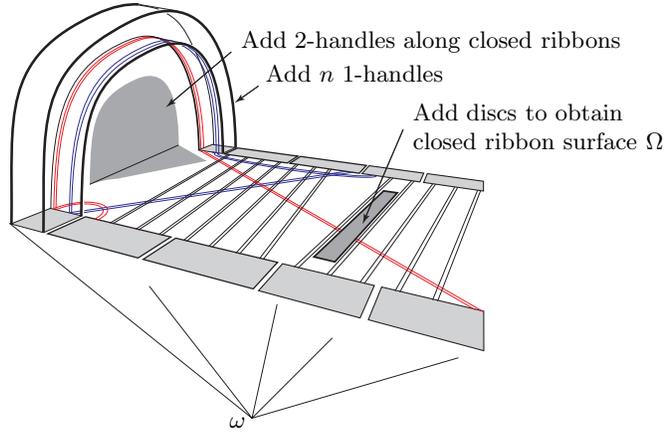}

            }   \caption{ \label{fig:2}
Surfaces and handles attached. The surface $\Omega$ is formed from the shaded discs, and coned to the pint $\omega$.}
   
\end{figure}

The description above involves a number a choices. 
We specify a choice 
as follows.
The ordering on the generators and relators enables us to associate a uniquely defined choice for $\Omega$ with a concrete realisation.
For technical convenience, we present the details under the assumption that all
relators have length greater than one.
 To each generator we take a pair of rectangles in the plane, arranged as two rows of $n$ rectangles. Each rectangle has a number of legs  determined by the total number of occurrences of $x_i$ in $r_1r_2\dots r_k$. The legs are attached 
to a horizontal edge. To these legs ribbons are then attached is the following way.  
In the listed relations for $P_G$, label occurrences of $x_i^\epsilon$ from left to right, from $1$ to $d_i$.
Let $d_{ij}$ denote the sum of magnitudes of exponents of all occurrences of $x_i$ in $r_j$.
Let 
$$d_i := \sum_{j=1}^k d_{ij}\qquad \qquad  
l_j := \sum_{i=1}^nd_{ij}\qquad \qquad
d := \sum_{i=1}^n   d_i \qquad \qquad  
l := \sum_{j=1}^kl_{j}
$$
Then $l_j$ is the length of $r_j$, and $d = l$ is the {\it length} of the presentation.
Accordingly we can write the relations in the  presentation as a concatenation of symbols of the form
 $ x_{i,p}^{\epsilon_{i,p}},\ \epsilon_{i,p} \in\{ -1, 1\} ,\ 1\leqslant p \leqslant d_i, \ 1\leqslant i \leqslant n$.
There is a ribbon from the leg labelled $x_{i,p}^{\epsilon_{i,p}}$ to the leg labelled $x_{j,q}^{-\epsilon_{j,q}}$ precisely when 
$x_{i,p}^{\epsilon_{i,p}}$ then $x_{j,q}^{\epsilon_{j,q}}$ appear
consecutively (possibly cyclically) in some relator.

\begin{figure}[htb] 
  \begin{center}
\psfrag{a}{$x_1$}
\psfrag{b}{$x_{2}$}
\psfrag{1}{1}
\psfrag{2}{2}
\psfrag{3}{3}
\psfrag{4}{4}
\psfrag{5}{5}
\psfrag{6}{6}
\psfrag{7}{7}
\psfrag{8}{8}
\includegraphics[width=2in]{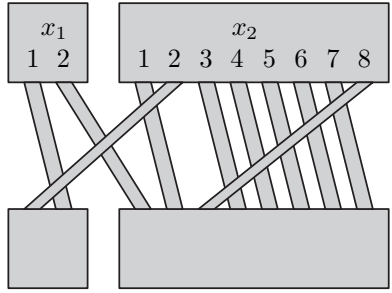}
\rule{0mm}{2cm}
\caption{\label{fig:Z6}
The presentation $\< x_1,x_2 \| x_1^2x_2^2=1, x_2^6=1\>$ of $\Z_6
\mathop{*}_{\Z_3}\Z_6$   has genus 1. Construction of the ribbon surface $\Omega$ is illustrated above.
The labels in the boxes are $x_{1,1}x_{1,2}$, and $x_{2,1}x_{2,2}x_{2,3}x_{2,4}x_{2,5}x_{2,6}x_{2,7}x_{2,8}$.
Corresponding to the first relator there are ribbons joining:
$x_{1,1}^{+}$ to $x_{1,2}^{-}$, $x_{1,2}^{+}$ to $x_{2,1}^{-}$,
 $x_{2,1}^{+}$ to $x_{2,2}^{-}$ and $x_{2,2}^{+}$ to $x_{1,1}^{-}$.
 Similarly, there are six ribbons corresponding to the second relator.
Notice that the presentation
$\< a,b \| a^2b^2=1, b^2=1\>$ has the same genus.}
\end{center}
\end{figure}

Let $\Omega$ ( or $\Omega_P$) be the closed surface obtained from the above constructed ribbon surface.
The surface  is not necessarily connected. Clearly this is the case if  the presentation is obviously a free product, with some generator not appearing in any relation. Note that by changing the presentation it can always be made connected.

\bigskip

\noindent
{\bf Lemma:} Every finitely presentable group has a presentation $P$ with   $\Omega_P$ connected.

\bigskip

\noindent
\begin{proof} Given an arbitrary finite  presentation $P_G$ as above, we may add a new generator $x_{n+1}$ and a new relation $x_{n+1}= x_1^2x_2^2 \dots x_n^2$. 
This  Tietze transformation and gives a new presentation for the same group. Since every generator occurs at least twice in the added relation, its corresponding annulus 
ensures that $\Omega$ is connected. Observe that given a presentation with a disconnected surface of genus $g$, we can obtain a presentation with a connected surface of genus $g$.
\end{proof}

From the construction above, and the observation of the preceding Lemma, we conclude the following.

\begin{thm}\label{thm:realisable}
Every finitely presentable group is the fundamental group of a compact singular 3-manifold having at most one singular point.
\end{thm}

If the coned-off boundary component, $\Omega$ is a 2-sphere, then the  
complex is a 3-manifold.
Since not all finitely presentable groups are 3-manifold groups, it is natural to ask
for the minimum achievable genus of $\Omega$. This gives a measure of how far the group is from being a 3-manifold group.

Given any closed orientable surface, possibly disconnected, by the genus of the surface 
we mean
the sum of the genera of its connected components.

\begin{defn}
The \emp{genus of a finite presentation} $P$ is the 
genus of the the surface $\Omega_P$ (connected or not) constructed above.
The \emp{$T$-genus of a group $G$} is then given by minimising over all finite presentations of $G$.
\end{defn}

Any two presentations are related by a finite sequence of Teitze transformations. 
For a fixed presentation 
the $T$-genus will not necessarily be equal to the link-genus same since there is 
extra restriction on the cyclic ordering about a vertex. The link-genus is clearly bounded above by the $T$-genus.

By rearranging the order of attachment of ribbons to each pair of discs, it is possible that a
ribbon surface  $\Omega_P'$ of genus higher or  lower than $\Omega_P$ may be obtained. We call such a rearrangement a  
 \emp{shuffle}, enabling us to     arrive at our final definition.

\begin{defn}
The \emp{genus of a group $G$} is defined to be the minimum genus achievable
for $\Omega'_P$ over all choices of presentation for $P$ and all shuffles.
\end{defn}

Note that the genus is obtained from the $T$-genus by allowing reordering
of ribbon attachments, simultaneously on each pair of discs. The link genus  
allows uncorrelated reordering. Thus

$$ 
\emph{link\ genus}(G)\  \leq \  \emph{genus}(G)\  \leq \  T{\rm -}\emph{genus}(G) \ .
$$

As an example, for the presentation of $\Z_6
\mathop{*}_{\Z_3}\Z_6$ in Figure \ref{fig:Z6}, the link genus is 0.

\begin{thm} \label{thm:3man}
A finitely presented group has genus zero if and only if it is the fundamental group of a compact, connected orientable 3-manifold.
\end{thm}

\begin{proof}
Since the genus of an orientable surface  is bounded below by 0, and the cone on a 2-sphere is a 3-dimensional ball, a group of genus 0 is clearly the fundamental group of a compact orientable 3-manifold. 

Conversely, if $M$ is a connected, compact orientable 3-manifold (possibly closed), delete an open ball neighbourhood of an interior point with boundary 2-sphere  $\Omega$. A handlebody structure exists for $M$ which consist of taking $\Omega \times I$, attaching some number $n$ of 1-handles to $\Omega \times  \{ 1 \}$, and then attaching $k$ 2-handles and a number of 3-handles. Attaching the 3-handles does not change the fundamental group. 
The $n$ 1-handles give $2n$ discs on the 2-sphere; the attaching circles for the 2-handles give rise to a collection of ribbons on the 2-sphere attached to the boundaries of the 1-handle attaching-discs. Hence the union of discs and ribbons defines a genus 0 surface.

On the other hand, an application of the Seifert-van Kampen Theorem gives a presentation of $\pi_1(M)$ with $n$ generators and $k$ relations, which we can 
order in any way we choose. However, the construction applied to the presentation 
yields a ribbon surface whose ribbon-attaching order to each of the 1-handle disc boundary circles may differ by some permutation, giving a surface of possibly higher genus.  This is why we allow the ribbons
to be reattached by arbitrary shuffles: it is unclear to the authors whether or not Tietze transformations alone suffice.

\end{proof}

 Similarly, it is clear that a group of genus $g$ can be realized as the fundamental group of a compact orientable 3-manifold with a connected boundary surface $\Omega$ of genus $g$ coned to a point.


\section{Examples}

\subsection{Genus 0}

If $G$ is the group of a closed non-orientable surface, the construction based on a standard 1-vertex polygon presentation yields a twisted I-bundle over the surface, which is an orientable 3-manifold, and hence of genus 0. 

\begin{qn}
 What is the genus of the fundamental group of a non-orientable   3-manifold?
\end{qn}

\subsection{ Groups of genus 1  }

It is well known that not all finitely presented groups are 3-manifold groups and hence there exist finitely presentable groups having genus at least one.
Clearly the next most important class to study, after genus 0 groups,
is the class of genus 1 groups. Since these are the quotients of 3-manifold groups by the normal closure of a peripheral $Z\oplus Z$ subgroup, there is some hope in adapting 3-manifold techniques to their study.





  {\bf  Baumslag-Solitar groups:} 
%
%
%
The presentation $\< x_1, x_2 \mid x_1x_2^mx_1^{-1}x_2^{-n}\>$ has genus 1 (unless
$|m|=|n|=1$). The construction for $m=2$, $n=3$ is shown in Figure \ref{fig:BS}.
\begin{figure}
\psfrag{a}{$a$}
\psfrag{b}{$b$}
\begin{center}
\includegraphics[height=2in]{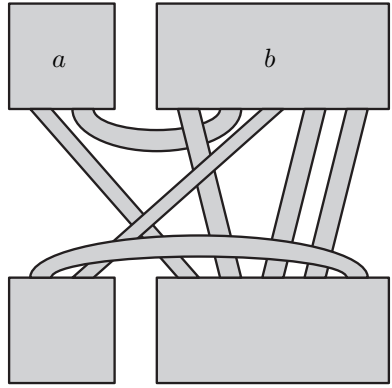}
\caption{\label{fig:BS} The surface for the presentation
$\<a,b\| ab^2a^{-1}b^{-3}=1\>$ has genus 1.}
\end{center}
\end{figure}

Concerning the existence of a group of genus 1,  observe that 
not all Baumslag-Solitar groups can be subgroups of a 3-manifold group \cite{Sh}. 
The following is an immediate consequence of this fact and Theorem \ref{thm:3man}.


\begin{thm}
 There exist groups of genus 1.\qed
  \end{thm}




%


\textbf{Example: $Z_6*_{Z_3}Z_6$} 

Consider the following presentations of $Z_6*_{Z_3}Z_6$:

$$
P=\langle a, b\| a^6=1,\ b^6=1,\ a^2b^{-2}=1\> \qquad\mbox{and}\qquad
Q=\langle  a,  b \ | \  a^6=1,\  a^2b^2 =1  \ \rangle.$$
We claim that $P$ has genus 2 and $Q$ genus 1.
Observe that we can always reduce exponents, provided their magnitude is greater than 2, to obtain a new presentation (of a possibly different group) having the same genus.
Hence the genera of $P$ and $Q$ are equal to the (presentation) genus of, respectively:
$$
P'=\langle a, b\| a^2=1,\ b^2=1,\ a^2b^{-2}=1\> \qquad\mbox{and}\qquad
Q'=\langle  a,  b \ | \  a^2, \ a^2b^2 =1 ,\ \ b^6 = 1 \ \rangle.$$
 Calculation shows that these presentations have genus 2 and 1 respectively.

\section{Triangulated complexes representing finitely presented groups}

Every closed (orientable) surface can be obtained by pairwise identification of edges of an even-sided polygonal 2-disc, and such a choice can be made with all vertices identified to a single vertex. Every compact surface with boundary can be similarly obtained when we allow deletion of a disc neighbourhood
of each of the vertices obtained after identification. The corresponding fundamental groups so obtained are respectively all closed surface groups, and all finitely generated free groups.

One  dimension higher, an analogue of a polygon is a 3-ball with triangulated boundary 2-sphere. Let $B_\tau$ denote such a ball with boundary triangulation $\tau$.    Recall that any triangulation of an orientable surface necessarily has an even number of triangles. 
Hence we may identify pairs of boundary triangles to obtain a quotient space, the closed orientable singular 3-manifold denoted ${\tilde B_{ \tau}} := B_\tau/$\ \   \llap {$\sim$}. This is a genuine 3-manifold if and only if its Euler characteristic is zero; equivalently, if and only if each vertex has link a surface of genus 0. Note that each pair of triangles can be identified in three possible ways, leading to $3^n(2n!)/2^nn!$ (possibly different) constructions from a triangulation with $2n$ triangles.

\begin{thm}\label{thm:polyhedra}
Every finitely presentable group is the fundamental group of a compact singular 3-manifold obtained from some ${\tilde B_{ \tau}}$ by deleting a neighbourhood of all but one vertices. 
Every  finitely presentable group which is a compact 3-manifold group is the fundamental group obtained from  some ${\tilde B_{ \tau}}$ by deleting a neighbourhood of all  vertices. 
  \end{thm}

\begin{proof} Given a finite presentation $P$ for some group $G$, we obtain from the construction of Section 2 a 3-manifold $M_P$ with boundary components $\Omega_i,\ i= 0,1,\dots , b$, with one boundary component $\Omega_0$ coned to a  point $\omega_0$. Every compact 3-manifold can be triangulated as a simplicial complex, inducing a triangulation of each boundary component. Now, cone each other boundary component $\Omega_i, \ i\not= 0$,   to its own   distinct point $\omega_i$. The resulting singular 3-manifold $M_*$ is a union of tetrahedra with pairwise boundary triangle identifications. Every vertex in this complex has a neighbourhood which is a cone on an orientable surface, necessarily a 2-sphere unless the surface is a boundary component $\Omega_i$ of the constructed singular 3-manifold $M_P$. Deleting vertices with 
ball neighbourhoods does not effect the fundamental group; deleting vertices $\omega_i, \ i\not=0$,  yields a space with the given fundamental group; deleting all vertices yields an orientable 3-manifold.

Choose a maximal tree in the dual 1-skeleton;
the tetrahedra corresponding to the vertices of the tree glue together to create a triangulated ball, with induced triangulation of its boundary sphere.  
 Thus we find that the complex  can be obtained as an identification space of a triangulated ball by pairwise identification of boundary faces. Deleting all $\omega_i, \ i\not= 0$, gives a space with fundamental group isomorphic to  $G$. 
\end{proof}

\begin{rem}
Every finitely presented group arises as the fundamental group of a singular 3-manifold obtained by partial pair identification of a triangulated ball.
\end{rem}

\section{Discussion and questions}  
 
\subsection{\bf Calculating genus}
 
 At this stage, we only know that groups of genus 0 and genus 1 exist. 
 Given a presentation, we know that
the genus of the presentation is bounded above by
$\frac{1}{2}(l+1)-n$ (where $n$ is the number of generators and $l$ is the length of the presentation), and that we may assume that 
all powers can by reduced mod 2 to be of the form $x_i^k,\ k= \pm1, \pm 2$.
Note that we do not claim that $x_i^{-1}$ can be replaced by $x_i^1$, or that $x_i^2$ can be replaced by $x_i^0$. It is also clear that there exist \emph{presentations} of arbitrarily large genus involving only these restricted exponents.

{\bf Problem:}  Clarify the complexity of determining :
\begin{enumerate}
\item The existence of an algorithm to determine the genus of a group.  Given that a group has genus at most 2, how do we determine that it cannot have genus 1? Compact 3-manifold groups, those of genus 0,  may enjoy certain properties, such as residual finiteness, which other groups may not. For example, all Baumslag-Solitar groups have genus at most 1, but not all are residually finite. We know of no invariant which prevents  a group from having genus at most 1. The classification of surfaces by Euler characteristic or curvature properties, manifest in properties of their own fundamental groups,  suggests there may be different properties of groups of genus 0, genus 1, and genus greater than 1;

 \item The existence of an algorithm to determine whether two groups of the same genus are isomorphic. Note that Perelman's work resolving Thurston's Geometrization Conjecture  pertains to the effective enumeration classification of 
compact 3-manifolds, and hence of genus 0 groups.

\end{enumerate}

The isomorphism problem for finitely presentable groups is algorithmically unsolvable. Hence one or both of the two preceding problems must be unsolvable. We believe that the first is most likely unsolvable, and that there may be  some chance that an algorithm may exist to solve the second problem. This will involve understanding the extent to which known 3-manifold techniques extend to the singular case.

Concerning the first problem, it is clear that the difficulty lies in the  combinatorics
of the construction, related to Tietze transformations. Reordering generators leaves genus unchanged, but reordering relations may do so, but  computably. Generally, Tietze transformations involving reordering,  duplication of a relation, taking an inverse, conjugation of a relation by a given generator, or replacing a relation by its product with another are also operations for which an algorithm to determine genus should be easy to find. It seems to the authors that the main difficulty lies in the addition or deletion of a relation involving some arbitrary word in given generators.

\bigskip

 Topological constructions on surfaces also give insight into different presentations 
 arising from normal forms of 3-manifold spines: a generic spine of a 3-manifold
 is a 2-complex with edges of degree 3, and with the link of any vertex forming the 1-skeleton of a tetrahedron. The following can be easily proved using Tietze transformations;
 what is of interest is that the genus of the resulting form of presentation does not change.
 
 \begin{thm}
Suppose $G$ is a finitely-presentable group with presentation of genus $g$.  Then $G$ admits a presentation yielding a surface of genus $g$
 such that every generator
    appears {exactly} three times among all relations. 
 \end{thm}
  
  \begin{proof}Take a finite presentation, and construct 
the 3-manifold with (possibly disconnected) distinguished surface $\Omega_0$, which is coned to $\omega$. The surface
 $\Omega_0$ contains the vertex ribbon graph consisting of pairs of discs with ribbons attached, filling the surface. 
 Normalize  the discs   by taking   unit radius discs in the plane centred at the points 
 $(1,\pm 2),\ (2,\pm 2),\ , \dots (n,\pm 2)$. 
 If generator $x_i$ has degree $d_i$, consider the $d_i$ points 
 $e^{2\pi i. k/d_i},\ k = 1, \dots , d_i, $ on the unit circle at
 $(i, -2)$. Replace the disc pair $(i, \pm 2)$ by $2d_i$ disc pairs centred at the points $e^{2\pi i. k/d_i},\ k = 1, \dots , d_i, $ on the pair of boundary circles. 
 
 This allows us to see that the group can be given a new presentation with generators
 $x_i^j  =1, \dots , n,\ j = 1, \dots , d_i$, with each occurrence of 
 $x_i^{\pm 1}$ in a relation as $(x_i^j)^{\pm 1}$ 
 together with new relations $x_i^j(x_i^{j-1})^{-1}$. Thus every 
 generator  $x_i^j$ occurs exactly 3 times.
\end{proof}

 Topologically, we have modified the structure of each attached 1-handle: the ribbons passing over the handle subdivide its boundary  annulus into a number of squares, representing new relations, and its two boundary circles into a number of arcs, giving new
 ribbons connecting, and vertices giving centres for new 1-handle attaching discs. A 1-handle becomes $d_i$ 1-handles and $d_i$ 2-handles; the two attaching discs become two boundary circles for the new ribbon surface.
 
 At the graph level, an embedded graph on a surface canonically gives rise (by 	`truncation') to an embedded 3-valent graph by replacing each vertex of degree $d>3$ by a $d$-gon, with edges of the original graph terminating at the vertices of the polygon.

  This has the following consequence:
 Take an arbitrary finite collection of polygons,
 with total number of edges a multiple of 3. Partition the edges arbitrarily into triples of edges, and assign orientations arbitrarily to all edges. Choose a generator for each triple, and label the corresponding edges. Construct a presentation
 with these generators, and with relations obtained by reading the cyclically ordered edges of polygons. All finitely presentable groups arise in this way in infinitely many ways.

\bigskip

Such a construction gives a presentation which is reminiscent of   the `utilities' problem of connecting 3 services to 3 dwellings (a standard example of the failure of $K(3,3)$ to 
be embeddable in the plane), as illustrated in Figure \ref{fig:utilities}.

\begin{figure}[htbp] 
   \centering
   \includegraphics[width=4in]{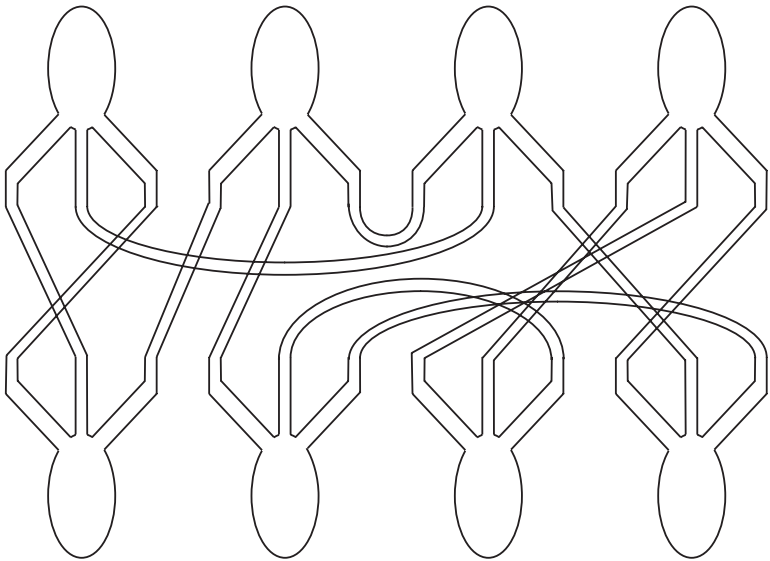}
   \caption{Regular degree 3. The surface has Euler characteristic $8 - 12 +2 = -2$: 
genus 2. One-relator presentation:}
   $\langle x_1,x_2,x_3,x_4 | 
x_1^2x_3^{-1}x_4^{-2}x_2x_3^{-1}x_4^{-1}x_3^{-1}x_2^2x_1\rangle$
   \label{fig:utilities}
\end{figure}

When genus was defined earlier, we raised the question of whether or not all shuffles were required to realize the minimal genus, or whether Tietze moves sufficed. Shuffles correspond to elements of the 
  group $S_{d_1}\times \dots \times S_{d_n} \subset S_d$,  where the symmetric group $S_{d_i}$ acts naturally on the set of points at which ribbons are connected to the attaching discs for the $i$th 1-handle. When each $d_i = 3$, considerable simplification occurs.  Generally these points are naturally partitioned into subsets corresponding to relations, since
  $d_i = d_{i,1} + \dots + d_{i,k}$, as illustrated in Figure  \ref{fig:bridson}. Tietze moves respect this partitioning, making it unclear whether the full set of shuffles is necessary. In a similar vein, 
restricted classes  or combinations of Tietze transformations, such as the Nielson moves related to the Andrews-Curtis Conjecture for balanced
presentations of the trivial group, could be used to refine a notion of genus.

\begin{figure}[htbp] 
   \centering
   \psfrag{a}{$a_{1}$}
   \psfrag{c}{$a_{2}$}
   \psfrag{b}{$b_{1}$}
   \psfrag{d}{$b_{2}$}
   \includegraphics[width=4in]{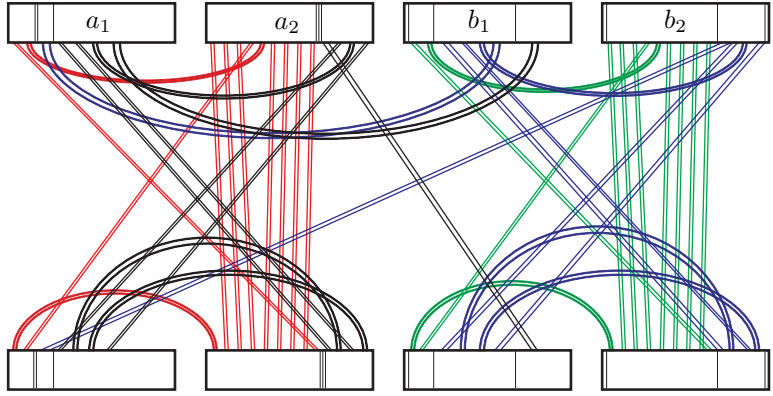} 
   \caption{
  Bridson et al.'s examples of an aspherical group with no finite quotients, $p=4$.
 $ \langle  \ a_1, a_2, b_1,b_2 
\ | \ 
a_1^{-1}a_2^pa_1a_2^{-p-1},\ \ 
b_1^{-1}b_2^pb_1b_2^{-p-1},\ \ 
a_1^{-1}[b_2,b_1^{-1}b_2b_1],\ \ 
b_1^{-1}[a_2,a_1^{-1}a_2a_1] \ \rangle  
$
The attaching points for ribbons in the boundary of each disc $A_{\pm i}$ are naturally partitioned into $k$   subsets from left to right:
$d_i = d_{i,1} + \dots + d_{i,k}$. 
}
   \label{fig:bridson}
   \end{figure}~

Another approach is to understand how genus changes when we add or delete relations
to a given presentation (thereby possibly changing the group).
 A given presentation can be considered as having been formed from a 1-relator presentation, by successively adding relations. If any relation is in the normal closure of previously added ones, the group remains the same, but the genus may increase: given a 
 compact bounded surface, adding ribbons may increase or decrease the number of boundary components, but cannot decrease the genus.

\bigskip

\noindent
{\bf Plumbing Observation:} Consider two 1-relator groups, each with the same number  of generators:
$$
P_K = \langle \ x_1,\dots , x_n \  | \  r_1 \ \rangle ,\qquad 
P_L = \langle \ x_1,\dots , x_n \  | \  r_2 \ \rangle .
$$
Suppose these give rise to surfaces $\Omega_K,  \ \Omega_L$, each with $2n$ discs and some number of ribbons. 
Then the presentation 
$$
P_G := \langle \ x_1,\dots , x_n \  | \  r_1, \  \  r_2 \ \rangle 
$$
gives rise to the surface $\Omega_G$  obtained canonically by plumbing
 the surfaces $\Omega_K,  \ \Omega_L$: corresponding discs $K_{\pm i}, \ L_{\pm i}$ are identified with discs $G_{\pm i}$ by a homeomorphism which maintains disjointness of attaching arcs for ribbons, with arcs for $r_2$ attached   anticlockwise from those of $r_1$ attached to $G_{+i}$.
The   genus possibilities for the 2-relator case, in turn created by plumbing 1-relator groups,
each of which giving a lower bound for the genus.

\subsection{Refinements based on 3-manifold and surface concepts}

Given a finitely presentable group $G$, there are many compact singular 3-manifolds having $G$ as fundamental group, with 0, 1 or more singular points and some number of boundary components. 
This singular 3-manifold arises from an underlying compact 3-manifold by partioning its boundary components into two subsets, and coning each connected boundary component in one set to a distinct point. 
To each singular point we associate the genus of its link. We can thus define genus to be the sum of genera of all singular points. 
Note that we can always choose to have a single singular point; can we lower the genus by allowing more? 

 In addition to such a total genus, we can consider the number of singular points, the number of boundary components, and the genus of each of these. Moreover, we might also consider the incompressibility of each boundary component in the underlying compact 3-manifold complementary
to all singular points.

Every closed orientable 3-manifold can be obtained by surgery on a link in the 3-sphere: thus every closed 3-manifold group arises from a link-complement group by killing off one generator from each peripheral $\Z\oplus \Z$ subgroup. On the other hand, given a link complement in either $S^3$ or some other closed 3-manifold, we can create a closed singular 3-manifold by attaching a cone to each boundary torus, thereby killing each peripheral $\Z\oplus \Z$ subgroup.
Do all finitely presentable groups arise by this construction?

Similarly, we might require that the underlying 3-manifold have some geometric or other structure, such as being hyperbolic, or have some associated notion of complexity, and investigate properties of groups in terms of finer invariants of such structure.



From a complexity viewpoint, we mention that
Costantino and  Thurston have recently considered  a related converse question, that 
of defining a notion of complexity for 3-manifolds in terms of the complexity of 4-manifolds they bound \cite{CT}.

%
%

%

\bibliography{AR}
\bibliographystyle{plain}

\end{document}